\documentclass{amsart}

\usepackage{amsmath}
\usepackage{amssymb}
\usepackage{adjustbox}
\usepackage{etoolbox}
\usepackage{libertine}
\usepackage{makecell}
\usepackage{mathtools}
\usepackage{multirow}
\usepackage{nicematrix}
\usepackage{tikz-cd}
\usepackage{rubikcube}
\usepackage{verbatim}
\usepackage{wrapfig}
\usepackage{listings}
\usepackage{lstlang0}
\usepackage[normalem]{ulem}

\newcommand{\Xsub}[1]{\mathbin{\mathop{\times}\limits_{#1}}}



\definecolor{cite}{RGB}{44,123,182}
\definecolor{ref}{RGB}{215,25,28}

\newtheorem{theorem}{Theorem}
\newtheorem{definition}{Definition}
\newtheorem{proposition}{Proposition}

\newtheorem{example}{Example}
\newtheorem{lemma}{Lemma}

\newtheorem*{IGP}{Inverse Galois Problem}
\newtheorem{notation}{Notation}

\newcommand{\pare}[1]{\left( #1 \right)} 
\newcommand{\set}[1]{{\left\{ #1 \right\}}} 

\def\A{\mathbb{A}}
\def\X{\mathbb{X}}
\def\Q{\mathbb{Q}}

\def\Z{\mathbb{Z}}
\def\S{\mathbb{S}}
\def\F{\mathbb{F}}
\def\O{\mathcal{O}}

\def\Gal{{\operatorname{Gal}}}
\def\sgn{{\operatorname{sign}\,}}

\newcommand{\Cyc}[1]{ \Z/(#1) }
\newcommand{\Rub}{\mathcal{R}}

\newcommand{\Galf}[2]{ \Gal\left(#1,#2\right)}

\newcommand{\disc}[1]{\operatorname{disc}({#1})}

\allowdisplaybreaks

\providecommand{\keywords}[1]{\textbf{\textit{Keywords: }} #1}

\ExplSyntaxOn
\NewDocumentCommand{\cycle}{ O{\,} m }
 {
  (
  \alec_cycle:nn { #1 } { #2 }
  )
 }
\seq_new:N \l_alec_cycle_seq
\cs_new_protected:Npn \alec_cycle:nn #1 #2
 {
  \seq_set_split:Nnn \l_alec_cycle_seq { , } { #2 }
  \seq_use:Nn \l_alec_cycle_seq { #1 }
 }
\ExplSyntaxOff

\ExplSyntaxOn
\NewDocumentCommand{\npcycle}{ O{\,} m }
 {
  \npalec_cycle:nn { #1 } { #2 }
 }
\seq_new:N \npl_alec_cycle_seq
\cs_new_protected:Npn \npalec_cycle:nn #1 #2
 {
  \seq_set_split:Nnn \npl_alec_cycle_seq { , } { #2 }
  \seq_use:Nn \npl_alec_cycle_seq { #1 }
 }
\ExplSyntaxOff

\usepackage{tikz-3dplot}
\usetikzlibrary{3d}
\newif\ifshowcellnumber
\showcellnumbertrue

\definecolor{Y}{RGB}{255,255,110}
\definecolor{R}{RGB}{255,70,75}
\definecolor{G}{RGB}{151,216,56}
\definecolor{B}{RGB}{51,152,237}
\definecolor{W}{RGB}{255,255,255}
\definecolor{O}{RGB}{255,165,0}

\newcommand{\TikZRubikFaceLeft}[9]{\def\myarrayL{#1,#2,#3,#4,#5,#6,#7,#8,#9}}
\newcommand{\TikZRubikFaceRight}[9]{\def\myarrayR{#1,#2,#3,#4,#5,#6,#7,#8,#9}}
\newcommand{\TikZRubikFaceTop}[9]{\def\myarrayT{#1,#2,#3,#4,#5,#6,#7,#8,#9}}
\newcommand{\BuildArray}{\foreach \X [count=\Y] in \myarrayL%
{\ifnum\Y=1%
\xdef\myarray{"\X"}%
\else%
\xdef\myarray{\myarray,"\X"}%
\fi}%
\foreach \X in \myarrayR%
{\xdef\myarray{\myarray,"\X"}}%
\foreach \X in \myarrayT%
{\xdef\myarray{\myarray,"\X"}}%
\xdef\myarray{{\myarray}}%
}

\pgfmathsetmacro\radius{0.1}
\tdplotsetmaincoords{55}{135}

\numberwithin{equation}{section}

\begin{document}

\title{Rubik's as Galois': the inverse Galois problem for the Rubik's Cube group}

\author{M. Mereb}
\author{L. Vendramin}

\address[Mereb]{IMAS--CONICET and Depto. de Matem\'atica, FCEN, Universidad de Buenos Aires, Pab.~1, Ciudad Universitaria, C1428EGA, Buenos Aires, Argentina}
\email{mmereb@dm.uba.ar}

\address[Vendramin]{Department of Mathematics and Data
Science, Vrije Universiteit Brussel, Pleinlaan 2, 1050 Brussels, Belgium}
\email{Leandro.Vendramin@vub.be}


\begin{abstract}
    We prove that the Rubik's Cube group can be realized as a 
    Galois group over the rationals. 
\end{abstract}

\maketitle


\section{Introduction}

It is a fundamental question in mathematics that seeks to determine whether every finite group can be realized as the Galois group of some field extension over the rational numbers. More concretely:

\begin{IGP}
    Given a finite group $G$, determine whether there exists a Galois extension of $\Q$ whose Galois group is isomorphic to $G$. Moreover, if there is such a Galois extension, find an explicit polynomial whose Galois group is~$G$.
\end{IGP}

The problem was first explicitly stated by mathematicians in the late 19th and
early 20th centuries, notably by Hilbert. For example, using his irreducibility
theorem, Hilbert showed \cite{MR1580277} that over the rationals there are infinitely many Galois
extensions with Galois group isomorphic to the symmetric group $\S_n$ and the
alternating group $\A_n$. 
A particularly elegant result in this context is a theorem by Schur, which concerns the Galois group of the 
$n$-th Taylor polynomial of the exponential function
\[
1 + X + \frac{1}{2!} X^2 + \cdots + \frac{1}{n!} X^n.
\]
Schur showed that the Galois group of this polynomial is isomorphic to the alternating group 
$\A_n$ if $n$ is divisible by 4, and to the symmetric group $\S_n$ otherwise; see \cite{MR925984} for a proof. There are other  families of polynomials with Galois groups isomorphic to the symmetric group. Particularly important for this paper is the work of Nart and Vila \cite{zbMATH01458933,zbMATH01458934}, 
where it is 
proved that the polynomial $X^n - X - 1$ has Galois group $\S_n$. The Nart--Vila theorem 
was later generalized by Osada in \cite{MR873881}.

At the beginning of the 20th century, Noether noted that the Inverse Galois Problem could be solved for a finite permutation group using the Hilbert irreducibility theorem, provided that the field of fractions of the ring of invariants of the group is a rational function field \cite{MR1511893}. However, it was later shown by Swan in \cite{MR244215} that not all rings of invariants of finite groups are rational, which means that Noether's method does not apply to every finite group.

For finite abelian groups, the Inverse Galois Problem can be easily solved using cyclotomic extensions; see for example \cite[Section 14.5]{MR1138725}. 

In 1937 Scholz \cite{MR1545668} and Reichardt \cite{MR1581540}, independently, proved that $p$-groups of odd order
are realizable as Galois groups over the rationals. Concerning solvable groups,
Shafarevich proved that finite solvable groups occur as Galois groups over the
rationals \cite{zbMATH03118500}. The original proof contains a mistake concerning the prime
2, later corrected in \cite{zbMATH04114880}. For a complete proof, we refer to the book 
\cite{MR2392026}. We note that Shafarevich's proof is not constructive, and so
does not produce a polynomial having a prescribed finite solvable group as a
Galois group. 

The most successful approach to the Inverse Galois Problem to date is based on the concept of rigidity. While the term ``rigidity'' was introduced by Thompson in \cite{zbMATH03879038}, the underlying idea was anticipated in earlier works by Shih \cite{zbMATH03438994}, Fried \cite{MR453746}, Belyi \cite{MR534593}, and Matzat \cite{MR524980}. This powerful method has been employed to realize many families of groups. For example, many of the sporadic simple 
groups have been realized as Galois groups 
over the rationals, like four of the Mathieu groups
$M_{11}$, $M_{12}$, $M_{22}$ and $M_{24}$. Thompson 
proved that the Monster group, a simple group of order
\[
808017424794512875886459904961710757005754368000000000,
\]
can be realized as a Galois group over the rationals. 

Despite significant progress, the Inverse Galois Problem 
remains unsolved for many groups, making it a central topic in algebra and number theory.
It is still unknown, 
for example, 
for the Mathieu group $M_{23}$\footnote{In a recent preprint, $M_{23}$ is realized as a Galois group over $\Q$; see~\cite{M23}.}, and the same happens for most of the 
simple groups of Lie type. 

For an elementary introduction to the Inverse Galois Problem, see \cite{MR1405612}. For more advanced presentations, we refer to \cite{MR3822366} and \cite{MR2363329}. 



In this short note, we 
consider the case of the Rubik's Cube group $\Rub$. More precisely, we will prove 
that $\Rub$ is realizable as a Galois group 
over the rational numbers.



\begin{theorem}
\label{thm:main}
    Let 
    \[
    g(X) = X^{24} + \frac{3852443469645611961262219752967766016}{384257037754753807138505851908147025}
    (X^2 + 1)
    \]
    and  
    \begin{align*}
   f(X) = X^{24} &- 24X^{22} + 8X^{21} + 252X^{20} - 168X^{19} - 1484X^{18} - 627X^{17}\\
   &+26628X^{16}
   - 97918X^{15} + 199671X^{14} - 266679X^{13} 
   + 234997X^{12}\\
   &-114681X^{11} 
    - 10107X^{10} + 63686X^9
    - 45384X^8 + 6819X^7\\
    &+ 12880X^6 -
    12096X^5 + 5502X^4 
    - 1504X^3 + 252X^2 - 24X + 1.
    \end{align*}
    Then $f(X)g(X)$ has Galois group over the rationals isomorphic to the Rubik's Cube
    group $\Rub$.
\end{theorem}




The statement of Theorem \ref{thm:main} can be easily verified using the computational algebra system Magma \cite{zbMATH01077111}; here we use Magma V2.28-18.
The calculation takes only a few minutes on a standard desktop computer.



A \emph{parametric family} of Galois extensions is a family of field extensions depending on some parameters $t_1,\dots,t_n$ such that the Galois group of this extension over $\Q(t_1,\dots,t_n)$ is some given group $G$. Hilbert's irreducibility theorem implies that, for infinitely many values of $t_1,\dots,t_n$, the specialization has Galois group $G$ as well. Thus, a parametric family provides a way to construct numerous specific extensions of $\Q$ with Galois group $G$ by specializing the parameter $t$.

\begin{theorem}
    \label{thm:parametric}
There is a parametric family of 
$\Rub-$extensions over $\Q.$
\end{theorem}

Theorem \ref{thm:parametric} will be proved in Section \ref{section:parametric}. 

Before going into the proof of our main theorems, it is convenient to give a concrete presentation of $\Rub$. We 
present a permutation representation of $\Rub$ 
by numbering the squares as follows: 



\TikZRubikFaceLeft
        {O}{O}{O}
        {O}{O}{O}
        {O}{O}{O}        
\TikZRubikFaceRight
        {B}{B}{B}
        {B}{B}{B}
        {B}{B}{B}
\TikZRubikFaceTop
        {Y}{Y}{Y}
        {Y}{Y}{Y}
        {Y}{Y}{Y}
\BuildArray

\begin{minipage}[b]{0.4\textwidth}
\centering
\begin{tikzpicture}
 \clip (-3,-2.5) rectangle (3,2.5);
 \begin{scope}[tdplot_main_coords]
  \filldraw [canvas is yz plane at x=1.5] (-1.5,-1.5) rectangle (1.5,1.5);
  \filldraw [canvas is xz plane at y=1.5] (-1.5,-1.5) rectangle (1.5,1.5);
  \filldraw [canvas is yx plane at z=1.5] (-1.5,-1.5) rectangle (1.5,1.5);
  \foreach \X [count=\XX starting from 0] in {-1.5,-0.5,0.5}{
   \foreach \Y [count=\YY starting from 0] in {-1.5,-0.5,0.5}{
   \pgfmathtruncatemacro{\Z}{\XX+3*(2-\YY)}
   \pgfmathsetmacro{\mycolor}{\myarray[\Z]}
    \draw [thick,canvas is yz plane at x=1.5,shift={(\X,\Y)},fill=\mycolor] (0.5,0) -- ({1-\radius},0) arc (-90:0:\radius) -- (1,{1-\radius}) arc (0:90:\radius) -- (\radius,1) arc (90:180:\radius) -- (0,\radius) arc (180:270:\radius) -- cycle;

    \node[canvas is yz plane at x=1.5,shift={(-1,1)}]{9};
    \node[canvas is yz plane at x=1.5,shift={(0,1)}]{10};
    \node[canvas is yz plane at x=1.5,shift={(1,1)}]{11};
    \node[canvas is yz plane at x=1.5,shift={(-1,0)}]{12};
    \node[canvas is yz plane at x=1.5,shift={(0,0)}]{L};
    \node[canvas is yz plane at x=1.5,shift={(1,0)}]{13};
    \node[canvas is yz plane at x=1.5,shift={(-1,-1)}]{14};
    \node[canvas is yz plane at x=1.5,shift={(0,-1)}]{15};
    \node[canvas is yz plane at x=1.5,shift={(1,-1)}]{16};
    
    \pgfmathtruncatemacro{\Z}{2-\XX+3*(2-\YY)+9}
    \pgfmathsetmacro{\mycolor}{\myarray[\Z]}
    \draw [thick,canvas is xz plane at y=1.5,shift={(\X,\Y)},fill=\mycolor] (0.5,0) -- ({1-\radius},0) arc (-90:0:\radius) -- (1,{1-\radius}) arc (0:90:\radius) -- (\radius,1) arc (90:180:\radius) -- (0,\radius) arc (180:270:\radius) -- cycle;
    \node[canvas is xz plane at y=1.5,shift={(-1,1)},xscale=-1] {19};
    \node[canvas is xz plane at y=1.5,shift={(0,1)},xscale=-1] {18};
    \node[canvas is xz plane at y=1.5,shift={(1,1)},xscale=-1] {17};
    \node[canvas is xz plane at y=1.5,shift={(-1,0)},xscale=-1] {21};
    \node[canvas is xz plane at y=1.5,shift={(0,0)},xscale=-1] {F};
    \node[canvas is xz plane at y=1.5,shift={(1,0)},xscale=-1] {20};
    \node[canvas is xz plane at y=1.5,shift={(-1,-1)},xscale=-1] {24};
    \node[canvas is xz plane at y=1.5,shift={(0,-1)},xscale=-1] {23};
    \node[canvas is xz plane at y=1.5,shift={(1,-1)},xscale=-1] {22};

    \pgfmathtruncatemacro{\Z}{2-\YY+3*\XX+18}
    \pgfmathsetmacro{\mycolor}{\myarray[\Z]}
    \draw [thick,canvas is yx plane at z=1.5,shift={(\X,\Y)},fill=\mycolor] (0.5,0) -- ({1-\radius},0) arc (-90:0:\radius) -- (1,{1-\radius}) arc (0:90:\radius) -- (\radius,1) arc (90:180:\radius) -- (0,\radius) arc (180:270:\radius) -- cycle;

    \node[canvas is yx plane at z=1.5,shift={(-1,1)},xscale=-1,rotate=-90] {1};
    \node[canvas is yx plane at z=1.5,shift={(-1,0)},xscale=-1,rotate=-90] {2};
    \node[canvas is yx plane at z=1.5,shift={(-1,-1)},xscale=-1,rotate=-90] {3};
    \node[canvas is yx plane at z=1.5,shift={(0,1)},xscale=-1,rotate=-90] {4};
    \node[canvas is yx plane at z=1.5,shift={(0,0)},xscale=-1,rotate=-90] {U};
    \node[canvas is yx plane at z=1.5,shift={(0,-1)},xscale=-1,rotate=-90] {5};
    \node[canvas is yx plane at z=1.5,shift={(1,1)},xscale=-1,rotate=-90] {6};
    \node[canvas is yx plane at z=1.5,shift={(1,0)},xscale=-1,rotate=-90] {7};
    \node[canvas is yx plane at z=1.5,shift={(1,-1)},xscale=-1,rotate=-90] {8};    
    }
   }
 \end{scope}
\end{tikzpicture}
\end{minipage} 
\TikZRubikFaceRight
        {R}{R}{R}
        {R}{R}{R}
        {R}{R}{R}        
\TikZRubikFaceTop
        {G}{G}{G}
        {G}{G}{G}
        {G}{G}{G}
\TikZRubikFaceLeft
        {W}{W}{W}
        {W}{W}{W}
        {W}{W}{W}
\BuildArray
\begin{minipage}[b]{0.4\textwidth}
\centering
\begin{tikzpicture}[rotate=180,transform shape]
 \clip (-3,-2.5) rectangle (3,2.5);
 \begin{scope}[tdplot_main_coords]
  \filldraw [canvas is yz plane at x=1.5] (-1.5,-1.5) rectangle (1.5,1.5);
  \filldraw [canvas is xz plane at y=1.5] (-1.5,-1.5) rectangle (1.5,1.5);
  \filldraw [canvas is yx plane at z=1.5] (-1.5,-1.5) rectangle (1.5,1.5);
  \foreach \X [count=\XX starting from 0] in {-1.5,-0.5,0.5}{
   \foreach \Y [count=\YY starting from 0] in {-1.5,-0.5,0.5}{
   \pgfmathtruncatemacro{\Z}{\XX+3*(2-\YY)}
   \pgfmathsetmacro{\mycolor}{\myarray[\Z]}
    \draw [thick,canvas is yz plane at x=1.5,shift={(\X,\Y)},fill=\mycolor] (0.5,0) -- ({1-\radius},0) arc (-90:0:\radius) -- (1,{1-\radius}) arc (0:90:\radius) -- (\radius,1) arc (90:180:\radius) -- (0,\radius) arc (180:270:\radius) -- cycle;

    \node[canvas is yz plane at x=1.5,shift={(-1,1)},rotate=180]{40};
    \node[canvas is yz plane at x=1.5,shift={(0,1)},rotate=180]{39};
    \node[canvas is yz plane at x=1.5,shift={(1,1)},rotate=180]{38};
    \node[canvas is yz plane at x=1.5,shift={(-1,0)},rotate=180]{37};
    \node[canvas is yz plane at x=1.5,shift={(0,0)},rotate=180]{B};
    \node[canvas is yz plane at x=1.5,shift={(1,0)},rotate=180]{36};
    \node[canvas is yz plane at x=1.5,shift={(-1,-1)},rotate=180]{35};
    \node[canvas is yz plane at x=1.5,shift={(0,-1)},rotate=180]{34};
    \node[canvas is yz plane at x=1.5,shift={(1,-1)},rotate=180]{33};
    
    \pgfmathtruncatemacro{\Z}{2-\XX+3*(2-\YY)+9}
    \pgfmathsetmacro{\mycolor}{\myarray[\Z]}
    \draw [thick,canvas is xz plane at y=1.5,shift={(\X,\Y)},fill=\mycolor] (0.5,0) -- ({1-\radius},0) arc (-90:0:\radius) -- (1,{1-\radius}) arc (0:90:\radius) -- (\radius,1) arc (90:180:\radius) -- (0,\radius) arc (180:270:\radius) -- cycle;
    \node[canvas is xz plane at y=1.5,shift={(-1,1)},xscale=-1,rotate=180] {30};
    \node[canvas is xz plane at y=1.5,shift={(0,1)},xscale=-1,rotate=180] {31};
    \node[canvas is xz plane at y=1.5,shift={(1,1)},xscale=-1,rotate=180] {32};
    \node[canvas is xz plane at y=1.5,shift={(-1,0)},xscale=-1,rotate=180] {28};
    \node[canvas is xz plane at y=1.5,shift={(0,0)},xscale=-1,rotate=180] {R};
    \node[canvas is xz plane at y=1.5,shift={(1,0)},xscale=-1,rotate=180] {29};
    \node[canvas is xz plane at y=1.5,shift={(-1,-1)},xscale=-1,rotate=180] {25};
    \node[canvas is xz plane at y=1.5,shift={(0,-1)},xscale=-1,rotate=180] {26};
    \node[canvas is xz plane at y=1.5,shift={(1,-1)},xscale=-1,rotate=180] {27};

    \pgfmathtruncatemacro{\Z}{2-\YY+3*\XX+18}
    \pgfmathsetmacro{\mycolor}{\myarray[\Z]}
    \draw [thick,canvas is yx plane at z=1.5,shift={(\X,\Y)},fill=\mycolor] (0.5,0) -- ({1-\radius},0) arc (-90:0:\radius) -- (1,{1-\radius}) arc (0:90:\radius) -- (\radius,1) arc (90:180:\radius) -- (0,\radius) arc (180:270:\radius) -- cycle;

    \node[canvas is yx plane at z=1.5,shift={(-1,1)},xscale=-1,rotate=180] {46};
    \node[canvas is yx plane at z=1.5,shift={(-1,0)},xscale=-1,rotate=180] {44};
    \node[canvas is yx plane at z=1.5,shift={(-1,-1)},xscale=-1,rotate=180] {41};
    \node[canvas is yx plane at z=1.5,shift={(0,1)},xscale=-1,rotate=180] {47};
    \node[canvas is yx plane at z=1.5,shift={(0,0)},xscale=-1,rotate=180] {D};
    \node[canvas is yx plane at z=1.5,shift={(0,-1)},xscale=-1,rotate=180] {42};
    \node[canvas is yx plane at z=1.5,shift={(1,1)},xscale=-1,rotate=180] {48};
    \node[canvas is yx plane at z=1.5,shift={(1,0)},xscale=-1,rotate=180] {45};
    \node[canvas is yx plane at z=1.5,shift={(1,-1)},xscale=-1,rotate=180] {43};    
    }
   }
 \end{scope}
\end{tikzpicture}
\end{minipage}

The cube consists of 26 pieces. There are eight corner pieces with three different colored facets, twelve edge pieces, with two facets and two colors each, and six middle pieces, each having just one color.

There are six basic movements one can apply to the Rubik's Cube,
all of which leave each central square invariant. 
These transformations generate a permutation representation of the Rubik's Cube group $
\Rub$.
One can easily check with \cite{GAP4} or \cite{zbMATH01077111} that 
$\Rub$ has order 
\[
43252003274489856000=2^{27}3^{14}5^37^211.
\]
To prove Theorem \ref{thm:main}
we will use a description 
of this group that uses
wreath products; see Subsection \ref{subsection:Rubik}.

\bigskip 
One may ask where 
the motivation for 
Theorems \ref{thm:main} and \ref{thm:parametric} 
come from.
One afternoon, as usual, we were discussing math over a cup of coffee, specifically good exercises for a Galois theory course. During which time, Dino, the second author's ten-year-old son, was playing with a Rubik's Cube. At some point, the cube became the main topic of our chat, and one of us mentioned an old paper by Zassenhaus~\cite{Zassen82} about the \emph{Rubik's Cube as a tool in Galois theory}. Zassenhaus' beautiful paper, however, does not address the Inverse Galois Problem; instead, it focuses on the Rubik's Cube group as a tool for exploring fundamental concepts in elementary group theory. We went back to Zassenhaus' paper, and we naturally asked ourselves whether (and how) the Rubik's Cube group can be realized as a Galois group over the rational numbers, and hence this project was born.

\section{Preliminaries}

\subsection{Wreath products}

We briefly review wreath products of finite groups. For a non-empty finite set $S$ and a group $A$, let
$A^S$ be the set of maps $S\to A$. 
A direct calculation shows that 
$A^S$ with 
\[
(\alpha\beta)(s)=\alpha(s)\beta(s),\quad \alpha,\beta\in A^S,\;s\in S,
\]
is a group. 


Let $G$ be a group acting  
on the right on $S$, so there is a map 
\[
S\times G\to S,\quad 
(s,x)\mapsto s\cdot x,
\]
such that 
$s\cdot 1=s$ for all $s\in S$ and 
$s\cdot (xy)=(s\cdot x)\cdot y$ for all $s\in S$ and $x,y\in G$. Then 
$G$ acts on $A^S$ by 
\begin{equation}
\label{eq:action}
(x\cdot \alpha)(s)=\alpha(s\cdot x),\quad \alpha\in A^S,\;x\in G,\;s\in S.    
\end{equation}

\begin{definition}
\label{def:wreath_product}
The \emph{wreath product} 
$A\wr_S G$ is defined 
as the semidirect product $A^S\rtimes G$, where the action
of $G$ on $A^S$ is that of \eqref{eq:action}. 
\end{definition}

The product of $A\wr_S G$ is given by
\[
(\alpha,x)(\beta,y)=(\alpha(x\cdot \beta),xy).
\]
We write $A\wr G$ when the $G$-action is clear from the context. 

A particular case that will be important for us 
is when $S=\{1,\dots,n\}$, $A$ is a finite group and $G=\S_n$ is the symmetric group
on $n$ letters. 

The power of wreath products goes far beyond the Rubik's Cube. Other remarkable group-theoretical applications can be found, for instance, in \cite{MR2674854}. 


    

For an integer $n\geq2$, we write $\Cyc{n}$
to denote the cyclic (additive) group of order~$n$. 
\begin{notation}
Let $n,m\geq2$ and $G$ be a subgroup of $\,\S_m$.  We 
write $(\Cyc{n}\wr G)^ \circ$ to denote 
the kernel of the group homomorphism
\begin{align*}
\Cyc{n}\wr G \to \Cyc{n}, \quad
(x, \sigma) \mapsto \sum_{i=1}^m x_i,
\end{align*}
where $x=(x_1,\dots,x_m)$. 
\end{notation}


\subsection{Rubik's Cube group}
\label{subsection:Rubik}

For our purposes, we will use a particular representation of 
the Rubik's Cube group based on wreath products. Let 
\begin{equation}
\label{eq:small_model}
\begin{aligned}
\Psi\colon\pare{\Cyc{3}\wr \S_8} \times \pare{\Cyc{2}\wr \S_{12}}
&\to  \Cyc{3} \times \Cyc{2} \times \set{\pm 1},
\\
(x,\rho,y,\sigma) &\mapsto \pare{\sum_{i=1}^8 x_i , \sum_{i=1}^{12} y_i, \sgn(\rho)\sgn(\sigma)},
\end{aligned}
\end{equation}
where $x=(x_1,\dots,x_8)$ and $y=(y_1,\dots,y_{12})$. 

Let $G$, $H$ and $K$ be groups, and $f\colon G\to K$ and $g\colon H\to K$ be group homomorphisms. Recall that the \emph{fiber product} $G\times_{f,g}H$ of the groups $G$ and $H$ which is the
subgroup
\[
\{(x,y)\in G\times H:f(x)=g(y)\}
\]
of $G\times H$.

The \emph{Rubik's Cube group} is 
    the 
    kernel $\Rub$ of the homomorphism $\Psi$ of~\eqref{eq:small_model}, 
    namely
    $$
    \pare{\Cyc{3}\wr \S_8}^\circ \times_{\sgn} 
    \pare{\Cyc{2}\wr \S_{12}}^\circ,
    $$
    where the symbol $\times_\sgn$ indicates a fiber product with respect to both maps
    $$\pare{\Cyc{n}\wr \S_m} \to \{\pm 1\},\quad
    (x,\sigma)\mapsto \sgn(\sigma),
    $$
    for $(n,m)\in\{(3,8),(2,12)\}$. 

We now briefly explain the group homomorphism of \eqref{eq:small_model}. For more details, we refer
to the book \cite[Section 2.5]{bandelow2012inside}. The Rubik's Cube has eight 
corners (i.e., \emph{corner cubie}), each of them with three different positions. 
Any valid permutation on the cube 
sends 
corner facets 
to corner facets. 
Thus the facets of a corner cube belong to the cyclic group $\Cyc{3}$ of three elements. Since 
there are eight corner cubes, 
the orientation of any facet of a corner cube can be described by 
the factor 
$\Cyc{3}\wr\S_8$
in the domain of the 
map $\Psi$ of \eqref{eq:small_model}.

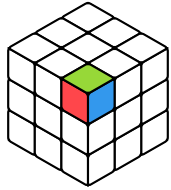
 \TikZRubikFaceLeft
         {W}{W}{R}
         {W}{W}{W}
         {W}{W}{W}        
 \TikZRubikFaceRight
         {B}{W}{W}
         {W}{W}{W}
         {W}{W}{W}
 \TikZRubikFaceTop
         {W}{W}{W}
         {W}{W}{W}
         {G}{W}{W}
 \BuildArray
\begin{center}
\begin{tikzpicture}[scale=.5, transform shape]
 \clip (-3,-2.5) rectangle (3,2.5);
 \begin{scope}[tdplot_main_coords]
  \filldraw [canvas is yz plane at x=1.5] (-1.5,-1.5) rectangle (1.5,1.5);
  \filldraw [canvas is xz plane at y=1.5] (-1.5,-1.5) rectangle (1.5,1.5);
  \filldraw [canvas is yx plane at z=1.5] (-1.5,-1.5) rectangle (1.5,1.5);
  \foreach \X [count=\XX starting from 0] in {-1.5,-0.5,0.5}{
   \foreach \Y [count=\YY starting from 0] in {-1.5,-0.5,0.5}{
   \pgfmathtruncatemacro{\Z}{\XX+3*(2-\YY)}
   \pgfmathsetmacro{\mycolor}{\myarray[\Z]}
    \draw [thick,canvas is yz plane at x=1.5,shift={(\X,\Y)},fill=\mycolor] (0.5,0) -- ({1-\radius},0) arc (-90:0:\radius) -- (1,{1-\radius}) arc (0:90:\radius) -- (\radius,1) arc (90:180:\radius) -- (0,\radius) arc (180:270:\radius) -- cycle;
   
    \pgfmathtruncatemacro{\Z}{2-\XX+3*(2-\YY)+9}
    \pgfmathsetmacro{\mycolor}{\myarray[\Z]}
    \draw [thick,canvas is xz plane at y=1.5,shift={(\X,\Y)},fill=\mycolor] (0.5,0) -- ({1-\radius},0) arc (-90:0:\radius) -- (1,{1-\radius}) arc (0:90:\radius) -- (\radius,1) arc (90:180:\radius) -- (0,\radius) arc (180:270:\radius) -- cycle;

    \pgfmathtruncatemacro{\Z}{2-\YY+3*\XX+18}
    \pgfmathsetmacro{\mycolor}{\myarray[\Z]}
    \draw [thick,canvas is yx plane at z=1.5,shift={(\X,\Y)},fill=\mycolor] (0.5,0) -- ({1-\radius},0) arc (-90:0:\radius) -- (1,{1-\radius}) arc (0:90:\radius) -- (\radius,1) arc (90:180:\radius) -- (0,\radius) arc (180:270:\radius) -- cycle;

    }
   }
 \end{scope}
\end{tikzpicture}
\end{center}

The cube also
has twelve edge cubes (i.e., \emph{edge cubie}), each of them having two positions:
Since there are twelve edge cubes, any facet of an edge cube belongs to 
the factor 
$\Cyc{2}\wr\S_{12}$ 
in \eqref{eq:small_model}.

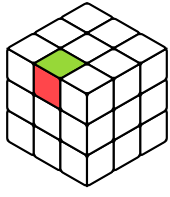
 \TikZRubikFaceLeft
         {W}{R}{W}
         {W}{W}{W}
         {W}{W}{W}        
 \TikZRubikFaceRight
         {W}{W}{W}
         {W}{W}{W}
         {W}{W}{W}
 \TikZRubikFaceTop
         {W}{W}{W}
         {G}{W}{W}
         {W}{W}{W}
 \BuildArray
 
\begin{center}
\begin{tikzpicture}[scale=.5, transform shape]
 \clip (-3,-2.5) rectangle (3,2.5);
 \begin{scope}[tdplot_main_coords]
  \filldraw [canvas is yz plane at x=1.5] (-1.5,-1.5) rectangle (1.5,1.5);
  \filldraw [canvas is xz plane at y=1.5] (-1.5,-1.5) rectangle (1.5,1.5);
  \filldraw [canvas is yx plane at z=1.5] (-1.5,-1.5) rectangle (1.5,1.5);
  \foreach \X [count=\XX starting from 0] in {-1.5,-0.5,0.5}{
   \foreach \Y [count=\YY starting from 0] in {-1.5,-0.5,0.5}{
   \pgfmathtruncatemacro{\Z}{\XX+3*(2-\YY)}
   \pgfmathsetmacro{\mycolor}{\myarray[\Z]}
    \draw [thick,canvas is yz plane at x=1.5,shift={(\X,\Y)},fill=\mycolor] (0.5,0) -- ({1-\radius},0) arc (-90:0:\radius) -- (1,{1-\radius}) arc (0:90:\radius) -- (\radius,1) arc (90:180:\radius) -- (0,\radius) arc (180:270:\radius) -- cycle;
   
    \pgfmathtruncatemacro{\Z}{2-\XX+3*(2-\YY)+9}
    \pgfmathsetmacro{\mycolor}{\myarray[\Z]}
    \draw [thick,canvas is xz plane at y=1.5,shift={(\X,\Y)},fill=\mycolor] (0.5,0) -- ({1-\radius},0) arc (-90:0:\radius) -- (1,{1-\radius}) arc (0:90:\radius) -- (\radius,1) arc (90:180:\radius) -- (0,\radius) arc (180:270:\radius) -- cycle;

    \pgfmathtruncatemacro{\Z}{2-\YY+3*\XX+18}
    \pgfmathsetmacro{\mycolor}{\myarray[\Z]}
    \draw [thick,canvas is yx plane at z=1.5,shift={(\X,\Y)},fill=\mycolor] (0.5,0) -- ({1-\radius},0) arc (-90:0:\radius) -- (1,{1-\radius}) arc (0:90:\radius) -- (\radius,1) arc (90:180:\radius) -- (0,\radius) arc (180:270:\radius) -- cycle;

    }
   }
 \end{scope}
\end{tikzpicture}
\end{center}

If we are allowed to take the cube apart and reassemble it, it follows
that there are as many movements as elements of the group
\[
\pare{\Cyc{3}\wr \S_8} \times \pare{\Cyc{2}\wr \S_{12}}.
\]
A \emph{position} of the cube is then a tuple
$(x,\rho,y,\sigma)$, where 
\begin{align*}
&x=(x_1,\ldots,x_8)\in (\Cyc{3})^{8}, 
&&
\rho\in\S_8,\\
&y=(y_1,\dots,y_{12})\in(\Cyc{2})^{12}, 
&&
\sigma\in\S_{12}.
\end{align*}

To understand the structure of the group $\Rub$, 
we refer to the following picture: 
\[
\Rub\simeq
  \vcenter{\hbox{ \TikZRubikFaceLeft    
         {R}{W}{R}
         {W}{W}{W}
         {R}{W}{R}
 \TikZRubikFaceRight
         {B}{W}{B}
         {W}{W}{W}
         {B}{W}{B}    
 \TikZRubikFaceTop
         {G}{W}{G}
         {W}{W}{W}
         {G}{W}{G}
 \BuildArray
\begin{tikzpicture}[scale=.5, transform shape]
 \clip (-3,-2.5) rectangle (3,2.5);
 \begin{scope}[tdplot_main_coords]
  \filldraw [canvas is yz plane at x=1.5] (-1.5,-1.5) rectangle (1.5,1.5);
  \filldraw [canvas is xz plane at y=1.5] (-1.5,-1.5) rectangle (1.5,1.5);
  \filldraw [canvas is yx plane at z=1.5] (-1.5,-1.5) rectangle (1.5,1.5);
  \foreach \X [count=\XX starting from 0] in {-1.5,-0.5,0.5}{
   \foreach \Y [count=\YY starting from 0] in {-1.5,-0.5,0.5}{
   \pgfmathtruncatemacro{\Z}{\XX+3*(2-\YY)}
   \pgfmathsetmacro{\mycolor}{\myarray[\Z]}
    \draw [thick,canvas is yz plane at x=1.5,shift={(\X,\Y)},fill=\mycolor] (0.5,0) -- ({1-\radius},0) arc (-90:0:\radius) -- (1,{1-\radius}) arc (0:90:\radius) -- (\radius,1) arc (90:180:\radius) -- (0,\radius) arc (180:270:\radius) -- cycle;
   
    \pgfmathtruncatemacro{\Z}{2-\XX+3*(2-\YY)+9}
    \pgfmathsetmacro{\mycolor}{\myarray[\Z]}
    \draw [thick,canvas is xz plane at y=1.5,shift={(\X,\Y)},fill=\mycolor] (0.5,0) -- ({1-\radius},0) arc (-90:0:\radius) -- (1,{1-\radius}) arc (0:90:\radius) -- (\radius,1) arc (90:180:\radius) -- (0,\radius) arc (180:270:\radius) -- cycle;

    \pgfmathtruncatemacro{\Z}{2-\YY+3*\XX+18}
    \pgfmathsetmacro{\mycolor}{\myarray[\Z]}
    \draw [thick,canvas is yx plane at z=1.5,shift={(\X,\Y)},fill=\mycolor] (0.5,0) -- ({1-\radius},0) arc (-90:0:\radius) -- (1,{1-\radius}) arc (0:90:\radius) -- (\radius,1) arc (90:180:\radius) -- (0,\radius) arc (180:270:\radius) -- cycle;

    }
   }
 \end{scope}
\end{tikzpicture}
}}
  \Xsub{\sgn}
  \vcenter{\hbox{ \TikZRubikFaceLeft
         {W}{R}{W}
         {R}{W}{R}
         {W}{R}{W}    
 \TikZRubikFaceRight
         {W}{B}{W}
         {B}{W}{B}
         {W}{B}{W}      
 \TikZRubikFaceTop
         {W}{G}{W}
         {G}{W}{G}
         {W}{G}{W}
 \BuildArray
\begin{tikzpicture}[scale=.5, transform shape]
 \clip (-3,-2.5) rectangle (3,2.5);
 \begin{scope}[tdplot_main_coords]
  \filldraw [canvas is yz plane at x=1.5] (-1.5,-1.5) rectangle (1.5,1.5);
  \filldraw [canvas is xz plane at y=1.5] (-1.5,-1.5) rectangle (1.5,1.5);
  \filldraw [canvas is yx plane at z=1.5] (-1.5,-1.5) rectangle (1.5,1.5);
  \foreach \X [count=\XX starting from 0] in {-1.5,-0.5,0.5}{
   \foreach \Y [count=\YY starting from 0] in {-1.5,-0.5,0.5}{
   \pgfmathtruncatemacro{\Z}{\XX+3*(2-\YY)}
   \pgfmathsetmacro{\mycolor}{\myarray[\Z]}
    \draw [thick,canvas is yz plane at x=1.5,shift={(\X,\Y)},fill=\mycolor] (0.5,0) -- ({1-\radius},0) arc (-90:0:\radius) -- (1,{1-\radius}) arc (0:90:\radius) -- (\radius,1) arc (90:180:\radius) -- (0,\radius) arc (180:270:\radius) -- cycle;
   
    \pgfmathtruncatemacro{\Z}{2-\XX+3*(2-\YY)+9}
    \pgfmathsetmacro{\mycolor}{\myarray[\Z]}
    \draw [thick,canvas is xz plane at y=1.5,shift={(\X,\Y)},fill=\mycolor] (0.5,0) -- ({1-\radius},0) arc (-90:0:\radius) -- (1,{1-\radius}) arc (0:90:\radius) -- (\radius,1) arc (90:180:\radius) -- (0,\radius) arc (180:270:\radius) -- cycle;

    \pgfmathtruncatemacro{\Z}{2-\YY+3*\XX+18}
    \pgfmathsetmacro{\mycolor}{\myarray[\Z]}
    \draw [thick,canvas is yx plane at z=1.5,shift={(\X,\Y)},fill=\mycolor] (0.5,0) -- ({1-\radius},0) arc (-90:0:\radius) -- (1,{1-\radius}) arc (0:90:\radius) -- (\radius,1) arc (90:180:\radius) -- (0,\radius) arc (180:270:\radius) -- cycle;

    }
   }
 \end{scope}
\end{tikzpicture}
}}. 
\]

By \cite[Theorem 1]{bandelow2012inside}, 
a position $(x,\rho,y,\sigma)$ is realistic (or possible) 
when it corresponds to a real configuration of the cube, and this 
happens if and only if the following conditions hold:
\begin{align*}
    &\sgn\rho =\sgn\sigma,\\
    &x_1+x_2+\cdots+x_8\equiv 0\bmod 3,\\
    &y_1+y_2+\cdots+y_{12}\equiv 0\bmod 2.
\end{align*}
Note that a real position corresponds exactly to a tuple that 
belongs to the orbit of the initial position 
of the cube. In particular, there are 
\[
43252003274489856000
\]
possible (legal) configurations 
for the Rubik's Cube.

For more information on the structure of group $\Rub$ we refer
to \cite[Section 2.5]{bandelow2012inside}. 







\TikZRubikFaceLeft
        {W}{W}{W}
        {W}{W}{W}
        {W}{W}{W}
\TikZRubikFaceRight
        {G}{G}{G}
        {G}{G}{G}
        {G}{G}{G}
\TikZRubikFaceTop
        {Y}{Y}{Y}
        {Y}{Y}{Y}
        {Y}{Y}{Y}
\BuildArray

\subsection{Number field background}

The following result is folklore. We include its proof for the reader's convenience.
\begin{lemma}\label{lem:lift}
Let $B$ be a commutative ring and $G \subseteq {\operatorname{Aut}} (B)$ be a finite group of ring automorphisms.
Consider $\phi, \psi : B \to D$ two ring homomorphisms from $B$ to an integral domain $D.$
Let $A = B^G$ be the subring of fixed elements and assume that $\phi$ and $\psi$ agree on $A.$
Then there exists $\sigma \in G$ such that $\phi = \psi \circ \sigma. $
\end{lemma}

\begin{proof}
We begin by proving a slightly weaker statement:
for each $b \in B$ there is a 
$\sigma \in G$ such that $\phi(b) = \psi(\sigma(b)).$
Let 
\[
f(X)=\prod_{\sigma \in G} (X - \sigma(b)) \in A[X],
\]
apply both $\phi$ and $\psi$ to its coefficients and observe that
the roots of the resulting polynomial $f^\psi(X) =  f^\phi(X) \in D[X]$ 
are $\{ \psi(\sigma(b))  \}_{ \sigma \in G}$. 
The claim follows since $b$ is one of the roots of $f(X)$. 

Suppose now that the lemma is false.
This means that for every $\sigma \in G$ there is an element 
$b_\sigma$ such that $\phi(b_\sigma) \neq \psi(\sigma(b_\sigma))$. 
Now apply the weaker claim already proved with
$B[ \set{X_\sigma }_{ \sigma \in G}]$,
$A[ \set{X_\sigma }_{ \sigma \in G}]$, and
$D[ \set{X_\sigma }_{ \sigma \in G}]$ 
instead of $B$, 
$A$ and $D$, respectively, 
the natural extensions of $\phi$ and $\psi$, 
the group $G$ acting on the coefficients of 
$B[ \set{X_\sigma }_{ \sigma \in G}]$, 
and 
$b=\sum_\sigma b_\sigma X_\sigma$. 
\end{proof}

A direct consequence of this lemma is the following well-known result: 

\begin{theorem}[Dedekind]
\label{thm:Dedekind}
Let $f(X)\in \Z[X]$ be a monic polynomial and $p\in \Z$ a prime not dividing its discriminant.
Then there is an element 
$\sigma \in \Galf{f(X)}{\Q}$ whose cycle type matches the degrees of the irreducible factors of 
$\overline {f}(X) \in \F_p[X],$ 
the reduction of $f(X)$ modulo~$p$. 
\end{theorem}

\begin{proof}
    Let $K$ be a splitting field of $f(X)$, $B=\O _ K$ its ring of integers and $G=\Gal(K/\Q)$. 
    Take $\mathfrak{P} \subseteq \O_K$ an ideal over $p\Z \subseteq \Z$. Let   
    $\psi\colon B\to B/\mathfrak{P}$ be the canonical map 
    and
    $\phi = {{\operatorname{Frob}}}_p \circ \psi$, where $\operatorname{Frob}_p$ denotes the Frobenius homomorphism, that is the map 
    $x\mapsto x^p$. 
    By Lemma~\ref{lem:lift}, there exists $\sigma \in G$ such that
    $\sigma( \mathfrak{P} ) \subseteq \mathfrak{P} $ and
    $\sigma(\alpha) \equiv \alpha^p \bmod{\mathfrak{P}}$ for all $\alpha \in B$ (this element
    $\sigma$ is called the  
    \emph{Frobenius element}). 

    The result follows by noting that the reduction modulo $p$ of $f(X)$, that is $\overline {f}(X) \in \F_p[X]$, is separable by the assumptions on the prime $p$ and the discriminant, and therefore the  
    ${{\operatorname{Frob}}}_p-$orbits of its roots in 
    an algebraic closure of $\F_p$ 
    are in correspondence with its irreducible factors.
\end{proof}

Similarly, we have the following 
specialization result:

\begin{proposition}\label{special}
Let 
$f(t,X) \in \Q(t)[X]$ be a monic separable polynomial. Consider $q \in \Q$ not a root of any denominator from the coefficients of $f$.
Let us also assume that $q$ is not a root of $\disc{f(t,X)} \in \Q(t)$. 
Then there is an embedding of $\Galf {f(q,X) }{\Q}$ in  
$\Galf{f(t,X)}{\Q(t)}$. 
\end{proposition}

\begin{proof}
    Consider $g(t)\in\Q[t]$ a common denominator for the coefficients of $f$ viewed as a polynomial in $X.$ 
    Let $K$ be the splitting field of $f$  over $\Q(t)$,
    $A = \Q[t, g^{-1}]\subseteq \Q(t)$  
    and $B \subseteq K$ be the ring of $A-$integral elements of $K.$
    Now pick any maximal 
    ideal  
    $\mathfrak{P} \subseteq B$ containing $t-q.$
    The specialization map 
    $A \to A/(t-q) = \Q$ extends to
    $ \psi : B \to B/\mathfrak{P}.$
    The \emph{decomposition group} 
    \[
    D_{\mathfrak{P}}=\{\sigma \in \Galf{f(t,X)}{\Q(t)} : \sigma(\mathfrak{P}) \subseteq  \mathfrak{P}\}
    \]
    acts on $B/\mathfrak{P}$ and fixes
    $A/(t-q).$ 
    Identifying $B/\mathfrak{P}$ with the splitting field of $f(q,X)$ over $\Q$ 
    we have a map 
    $D_{\mathfrak{P}} \to \Galf {f(q,X) }{\Q}.$ 
    It is surjective by Lemma~\ref{lem:lift}. The injectivity follows from the non-vanishing of $\disc{f(X)}$ at $t=q.$
\end{proof}

The following observation will be called the specialization argument.

\begin{lemma}[Specialization Argument]
\label{lem:spec}
Assume that $f(t,X)$ and $q$ satisfy the hypotheses of Proposition~\ref{special}. Assume that $\Gal(f(t,X), \Q(t))$ is a subgroup of $G$ and at $t=q$ the splitting field of $f(q,X)$ has degree exactly $|G|$. Then 
$$\Gal(f(q,X), \Q) \simeq 
G \simeq 
\Gal(f(t,X), \Q(t)).$$ 
\end{lemma}

\begin{proof}
    It follows immediately from Proposition~\ref{special}.
\end{proof}

Lemma~\ref{lem:spec} will be used to compute  Galois groups of parametric families by checking only one specialization, as in the next example. 


\begin{example}
\label{exa:X^n-tX-t}
    For $n\geq2$ consider  
    \[
  X^n - tX - t \in \Q(t)[X].
    \]
    By reducing modulo ${ ({t-1}) }$
(i.e., setting $t=1$) one gets  $ f(X)=X^n - X - 1 .$
    By \cite[Theorem 1]{SELMER1956}, the polynomial $f(X)$
    is irreducible over $\Q$, and, by \cite[Theorem 1]{MR873881},
    $\Galf{f(X)}{\Q}\simeq\S_n$.
Therefore, the polynomial $X^n - tX - t \in \Q(t)[X]$ has~$\S_n$
as Galois group over $\Q(t).$ 
\end{example}







\begin{example}\label{ex:A5}
The discriminant of 
\[
f(X) =  X^{5} + 20X + 16 
\]
is $2^{16}5^6$. It is a square, so 
$G=\Galf{f(X) }{\Q} \subseteq \A_5.$
Reducing modulo 7, 
we only find two roots, producing a $3-$cycle in $G.$ 
Also, its reduction modulo 3 is irreducible, producing a $5-$cycle. A subgroup of $\,\,\S_5$ containing both a $3-$
and a $5-$cycle must contain $\A_5.$
Therefore the Galois group of $f$ is $\A_5.$
\end{example}

\begin{example}\label{ex:wreathA5}
    Let us take $f(X^2)$ for $f(X)$ as in Example~\ref{ex:A5}.
    The roots of $f(X^2)$ are 
    $$
    \set{ \pm \sqrt{\alpha}\,: \,f(\alpha) = 0}.
    $$
    Then $\Galf{f(X^2)}{\Q}$ is isomorphic to a subgroup of 
    $\,\Cyc{2}\wr \A_5$. The blocks of imprimitivity of this Galois group are the five pairs 
    $\set{\sqrt{\alpha},-\sqrt{\alpha}}$,
    one per each root $\alpha$ of $f(X)$.
    Looking at $f(X^2)$ modulo 3, we see that it is irreducible. Reducing modulo many primes and producing permutations of different cycle types, one can check that the Galois group is actually isomorphic to the full group 
    $(\Cyc{2})\wr \A_5$, as in Example~\ref{ex:A5}.
\end{example}


\begin{example}\label{ex:subWreath}
    Consider instead 
    $-f(-X^2) = X^{10}+20X^2-16.$
    The roots of this polynomial are 
    $$
    \set{ \pm \sqrt{\beta}\,: \,f(-\beta) = 0}.$$
    Since the product of the roots of $f(-X)$ is a perfect square, 
    the group $\Galf{-f(-X^2)}{\Q}$ must preserve the product of the square roots.
    Therefore 
    \begin{align*}    
    \Galf{-f(-X^2) }{\Q}&\subseteq \set{ (z,\tau)\in (\Cyc{2})\wr \A_5 \,:\, \sum_{i=1}^5 z_i \equiv 0\bmod{2}}\\
    &=\pare{\Cyc{2}\wr \A_5}^\circ.
    \end{align*}
    Factorizing $-f(-X^2) $ modulo 3 and 7, we see that it is irreducible. Like in the previous example (namely, reducing modulo many primes), one can see that 
    $\,\Galf{-f(-X^2)}{\Q} $ is indeed the full $\pare{\Cyc{2}\wr \A_5}^\circ$.
\end{example}


\begin{example}\label{ex:CubiSubWreath}
In a similar vein, let us take now 
$ f(X^3) = X^{15} + 20X^3 + 16$ for $f(X)$ as in Example~\ref{ex:A5}. 
Proceeding as before, after adjoining a primitive cube root of unity 
$\omega$, we deduce
    $$
    \Galf{f(X^3)}{\Q [ \omega ]} 
    =\Cyc{3}\wr \A_5.
    $$
\end{example}




We shall need the following lemma, see for example 
\cite[page 41]{samuelalgnumthy}. 

\begin{lemma}
\label{lem:discriminant}
Let $f(X) = X^{n} -aX+b$. Then 
\[
\disc{f(X)} = (-1)^{n(n-1)/2}(n^{n }b^{ n-1}-(n-1)^{n-1 }a^{n }).
\]
\end{lemma}

\begin{proof}
    Let $\alpha_1,\dots,\alpha_n$ be the roots of $f(X)$. To compute the
    discriminant of $f(X)$, we use the well-known formula
    \[
    \disc{f(X)}=(-1)^{\frac{n(n-1)}{2}}\prod_{i=1}^nf'(\alpha_i).
    \]
    For $i\in\{1,\dots,n\}$, a direct calculation shows that
    $f'(\alpha_i)= 
    a(n-1)-\frac{bn}{\alpha_i}$. Moreover, 
    \[
    \prod_{i=1}^n f'(\alpha_i)= 
    n^nb^{n-1}-(n-1)^{n-1}a^n
    .
    \]
    From this, the claim follows. 
\end{proof}


\section{The proof of Theorem \ref{thm:main}}
\label{section:proof}



Keeping in mind that $\Rub$ is
isomorphic to a subgroup of
\[
\pare{\Cyc{3} \wr \S_8} \times \pare{\Cyc{2} \wr \S_{12}},
\]
we look for a couple of polynomials
$f_{24}$, and $g_{24}$ whose Galois groups are the subgroups of the wreath products. They are made out of polynomials $f_8$ and $g_{12}$ having Galois groups $\S_8$ and $\S_{12}$ respectively.
In addition, we need to choose the polynomials in such a way that the composite of both splitting fields has its Galois group embedded in the desired fibered product.

The subscripts in the name of the polynomials denote the degree.



\subsection{${{\rm sign}}-$fibered product}

We need to construct two splitting fields $E$ and $F$ of some polynomials 
    $f_8(X)$ and $g_{12}(X)$ over $\Q$  
    with Galois groups $\S_8$ and $\S_{12}$, respectively. 
    The condition on the Galois group as fibered product translates to a condition on their discriminants in the following way:
    $$
    \disc{f_8(X)} \disc{g_{12}(X)}
    \in (\Q^\times )^2.
    $$
    That is to say, the product of their discriminants must be a perfect square.

\subsection{$\Cyc{2}-$extensions}

Let us consider the polynomial
$g_{12}(X) = X^{12} + t(X+1)$. By Example~\ref{exa:X^n-tX-t}, 
its Galois group is $\S_{12}$ for almost any value of $t$. 
To get a Galois group embedded in $\Cyc{2}\wr \S_{12}$, 
one can consider $g_{24}(X)=g_{12}(X^{2})$, 
in analogy with Example~\ref{ex:wreathA5}.


We need to impose a condition on $t$ to ensure that 
$\Galf{g_{24}(X)}{\Q} \subseteq (\Cyc{2}\wr \S_{12})^\circ$. 
One way to achieve this is to require that the product of the 
roots of $g_{12}(X)$ be a perfect square (see Example~\ref{ex:subWreath}).
Thus take for example $g_{12}(X) = X^{12} + r^2(X+1)$ for some $r\in \Q$. 

For the fiber product, we need some control on the discriminant of $g_{12}(X)$
modulo $(\Q^{\times})^2.$
By Lemma~\ref{lem:discriminant},
\begin{align*} 
\disc{g_{12}(X)} & = 
(-1)^{ 12 \choose  2} 
\pare{12^{12}(r^2)^{11} -11^{11}(-r^2)^{12}  }
\\
& = \pare{r^{11}12^{6}}^2\left(1 - 11 \cdot \left(\frac{r\cdot 11^5}{12^6}\right)^2\right ) 
\end{align*}
from which we get 
\begin{align}\label{e:PellDisc}
 \disc{g_{12}(X)} \equiv 1 - 11 u^2  \bmod{(\Q^{\times})^2} 
\end{align}
for some rational number $u$. In other words, 
$\disc{g_{12}(X)}$ is representable by the quadratic form
$ v^2 - 11w^2$ over the rationals.

\subsection{$\Cyc{3}-$extensions}
\label{ss:cyc3}

A natural way to construct a $\Cyc{3}-$extension is to take a cube root, provided the base field contains $\omega$, a primitive cube root of $1$
(as in Example~\ref{ex:CubiSubWreath}).

In fact, by Hilbert's Theorem~90, these are the cubic Galois extensions with enough roots of unity.
This means one needs to have 
$\Q[\sqrt{-3}]$ 
in the base field.
Having this subfield inside the splitting field of $f$ is achievable by taking its discriminant congruent to $-3$ in 
$\Q^\times / (\Q^\times)^2.$ But in this case the Galois group would also permute $\omega$ and ${\omega}^{-1}= \overline{\omega}.$

One way to avoid this is to consider the parametric family of 
$\Cyc{3}-$extensions 
\[
X^3 -t X^2 +(t-3)X + 1 \in \Q(t)[X].
\]
This has the advantage that it does not need the presence of cube roots of unity.


Given an irreducible polynomial $f(X)$ with roots 
$\{\alpha =  \alpha_1, \alpha_2, \ldots \},$
we can get a $\Cyc{3}$ extension of $\Q[\alpha]$
by considering the polynomial
\[
\tilde{f}(X)=
(X(X-1))^{\deg f} f\left( \frac{X^3-3X+1}{X(X-1)} \right).  
\]

Note that 
$X^3 -\alpha X^2 +(\alpha-3)X + 1$
divides $\tilde{f}(X)$ in $\Q[\alpha][X]$. When it is an irreducible factor,  
any of its roots $\beta$ generates over $\Q[\alpha] $ a $\Cyc{3}-$extension.
The conjugates of $\beta$ over 
$\Q[\alpha]$ are
$$
\frac{ 1 }{ 1 - \beta}, 
\frac{\beta-1} {\beta }
\text{ and }
\beta.
$$

Thus, if
$$
f(X) = \prod_\alpha (X-\alpha),
$$
then
\begin{align*}
\tilde{f}(X)=
\prod_\alpha 
(X^3 -\alpha X^2 +(\alpha-3)X + 1).
\end{align*}

We are going to need the following
identity soon
\begin{align}\label{eq:ftilde}   
(-\omega X-\overline{\omega})^{\deg \tilde{f}}\tilde{f}
\left( \frac{X+1}{-\omega X-\overline{\omega}} \right) =
\prod_\alpha 
( ( -\alpha -3\overline{\omega})X^3 -\alpha-3\omega).
\end{align}

\subsection{Wreath product}

Let us consider $f_8(X) = X^8-t X-s .$ 
By a similar argument to the one in Example~\ref{exa:X^n-tX-t}
this polynomial will have $\Galf{f_8}{\Q(s,t)}\simeq \S_8.$ 
By Hilbert's irreducibility (see, for example, \cite[Chapter 1]{MR1405612}), 
almost every rational specialization of $s$ and $t$ gives
$\S_8$ as Galois group over $\Q.$

Considering thus
\begin{align}\label{eq:f24}
f_{24}(X)=
 (X^2-X)^8 f_8\left(\frac{ X^3-3X+1 }{X^2-X} \right) .
\end{align}

We need to impose conditions on $s$ and $t$ that guarantee that $\Galf{f_{24}}{\Q} $ is a subgroup of
$(\Cyc{3} \wr \S_8)^\circ$
(much like $g(0)$ a perfect square for $(\Cyc{2} \wr \S_{12})^\circ$).
Since after adjoining $\Q[\omega]$ to the base field one has that 
$\Cyc{3}-$extensions are precisely those given by the adjunction of a cube root, we can play the same game as before (after a suitable change of variables).
This can be accomplished as follows: 
Over $\Q[\omega]$ we can 
diagonalize the matrices
$
\begin{psmallmatrix}
0 & 1 \\ -1 & 1
\end{psmallmatrix}
$
and
$
\begin{psmallmatrix}
1 & -1 \\ 1 & 0
\end{psmallmatrix}
$
conjugating by
$
\begin{psmallmatrix}
1 & 1 \\ -\omega  & -\overline{\omega}
\end{psmallmatrix}
$. 
Therefore, the polynomial
\begin{align}\label{eq:ftildetrans}
(-\omega X-\overline{\omega})
^{24} f_{24}
\left( \frac{X+1}{-\omega X-\overline{\omega}} \right) 
\end{align}
has only monomials of degree multiple of $3$ (recall~\eqref{eq:ftilde}).

\subsection{Index $3$ subgroup}

Dividing~\eqref{eq:ftildetrans} by the leading coefficient, we end up with 
\begin{align*}
X^{24} 
+ & \frac{(-18 \omega - 21)t - 8s + 52488}{3\overline{\omega}t - s + 6561\omega}X^{21} + 
\cdots 
\\
& \cdots+\frac{(18\omega - 3)t - 8s + 52488}{ 3\overline{\omega} t - s + 6561\omega}X^3
+ \frac{3\omega t - s + 6561\overline{\omega} }{ 3\overline{\omega}t - s + 6561\omega},
\end{align*}                            
a monic polynomial in $X^3$ with coefficients in $\Q[\omega].$ Its Galois group over $\Q[\omega]$ 
embeds in 
$(\Cyc{3} \wr \S_8).$
One way to guarantee it embeds in $(\Cyc{3} \wr \S_8)^\circ$ is to ask for its constant term to be a perfect cube
(in analogy with  Example~\ref{ex:subWreath}).

                              
Therefore, we need to find $s,t \in \Q$ such that
$$
\frac{3\omega t - s + 6561\overline{\omega} }{ 3\overline{\omega}t - s + 6561\omega} 
$$
is a perfect cube in $\Q[\omega].$
Setting 
$$
3\omega t - s + 6561\overline{\omega}  = c(a+\omega b)^3
$$
with rational numbers $a,b,c$ we can solve for $s$ and $t.$





There are several parameter choices that yield an $f_8(X)$ with the desired Galois group. 
For example, choosing $a=1$, $b=-1$ and $c=1$ gives $s=-6558$ and $t= 2185$. 
For the fiber product with $\Galf{g_{24}(X)}{\Q}$, we need to find an $f_{8}(X)$ whose discriminant
is of the form $v^2-11w^2.$
This can be achieved in many ways. 
A quick brute force search gives us many examples.
For instance $a=1$, $b=2$ and $c=24$, 
leads to $t= 2139$ and $s=-6489$. 
Therefore,
$f_8(X)=X^8 -2139X + 6489$
with discriminant 
\[
\disc{f_8(X)} =          
3^8 \cdot 7^7 \cdot 1437417619559484462138047,
\]
which is of the form $v^2 - 11 w^2$ for
$v=106936663173678765$
and
$w=18262481960816352.$

We have to consider now
\[
r = 
\frac{12^6w}{11^5v}
=\frac{1962764241992810496}{619884697145165705}
\]
to get
$g_{24}(X) = X^{24} + r^2(X^2+1).$
Following the substitution~\eqref{eq:f24} we get
the polynomial $f_{24}(X).$ 
These $f_{24}(X)$ and $g_{24}(X)$
are the $f$ and $g$ from the statement of Theorem~\ref{thm:main}.

\subsection{Final step}
\label{ss:final_step}
In each of the cases above, we can easily check that the Galois group
of $f_{24}(X)g_{24}(X)$ is as big as possible restricted to the imposed constraints. Namely
\[
|\Galf{f_{24}(X)g_{24}(X)}{\Q}|
=43252003274489856000.
\]
Here is the Magma code: 
\begin{lstlisting}[language=Magma]
> P<x> := PolynomialRing(Rationals());
> f8 := x^8 -2139*x + 6489;
> f24 := P!Numerator((x^2-x)^8*Evaluate(f8, 
> (x^3-3*x+1)/(x^2-x)));
> r := 1962764241992810496/619884697145165705;
> g24 := x^24 + r^2*(x^2+1);
> #GaloisGroup(f24*g24);
43252003274489856000
\end{lstlisting}

\subsection{Other examples}

One can get other polynomials with the same method.

For instance, taking
$a=-18$, $b=-9$ and $c=2$, we get 
$f_8(X) = X^8 + 729X + 2187$. This polynomial has discriminant 
\[
3949085439326327289928812040905
=
3^{48}\cdot 5\cdot 269\cdot 36809, 
\]
with prime factors considerably smaller than $1437417619559484462138047$. Then 
\begin{align*}
f_{24}(X) = X^{24} &- 24X^{22} + 8X^{21} + 252X^{20} - 168X^{19} - 1484X^{18} \\
&+ 2241X^{17} +
    2250X^{16} - 11878X^{15} + 41931X^{14}\\
    &- 126147X^{13} + 234997X^{12} -
    255213X^{11} + 147633X^{10}\\
    &- 22354X^9 - 21006X^8 + 3951X^7 + 12880X^6 -
    12096X^5\\
    &+ 5502X^4 - 1504X^3 + 252X^2 - 24X + 1.
\end{align*}

Note that 
\[
3949085439326327289928812040905
=v^2-11w^2
\]
for $v = 1992257950336974$ and $w = 42646860008631$. 
In this case, 
\[
r = 
\frac{12^6w}{11^5v}
=\frac{225441792}{568026877}
\]
and 
\[
g_{24}(X) = X^{24} + \left(\frac{225441792}{568026877}\right)^2(X^2+1).
\]
The Galois group of $f_{24}(X)g_{24}(X)$ is isomorphic to 
$\Rub$. 




With $a=7$, $b=-15$ and $c=1$, one gets 
$f_8(X)= X^8 +123X - 1196$. Then
\begin{align*}
    f_{24}(X) = X^{24} &- 24X^{22} + 8X^{21} + 252X^{20} - 168X^{19} - 1484X^{18}\\
    &+ 1635X^{17} +
    3109X^{16} + 4278X^{15} - 44915X^{14}\\
    &+ 84511X^{13} - 65443X^{12} + 14833X^{11} -
    5873X^{10}\\
    &+ 30162X^9 - 30449X^8 + 4557X^7 + 12880X^6\\
    &- 12096X^5 +
    5502X^4 - 1504X^3 + 252X^2 - 24X + 1.
\end{align*}

In this case, the discriminant of $f_8(X)$ is the negative of a prime number, namely 
\[
-58727088785134974217580322839=v^2-11w^2
\]
for $v = 760938559245$ and $w = 73067632314568$. 

Therefore, taking 
\[
r=\frac{12^6w}{11^5v}=\frac{24242086778798112768}{13616657322774055}, 
\]
we get  
\[
g_{24}(X)=X^{24} + \left(\frac{24242086778798112768}{13616657322774055}\right)^2(X^2+1).
\]
Again, one gets that $f_{24}(X)g_{24}(X)$ has Galois group
isomorphic to $\Rub$. 






\section{A parametric family: Proof of Theorem \ref{thm:parametric}}
\label{section:parametric}

In searching for Galois extensions of a given group, finding a so-called \emph{parametric family} of such extensions is always more desirable. This was helpful in Section~\ref{section:proof} to impose extra constraints on the parameters.


In this section, we show how to get a parametric family of polynomials 
\[
h(u,v,X) \in \Q(u,v)[ X ]
\]
having Galois group isomorphic to $\Rub$. Then Hilbert's Irreducibility Theorem (see \cite[Chapter 1]{MR1405612}) implies that the same property is shared by all specializations of $h$ at rational pairs $(u,v)\in \Q^2$ outside of a \emph{thin subset} 
(in the sense of~\cite[\textsection 3.1]{MR2363329}).






Using the $\S_n$-family $X^n -t (X + 1) \in \Q(t)[X]$
(see Example~\ref{exa:X^n-tX-t}) one can take
\[
f(X) = X^8 -t (X + 1)\quad\text{and}\quad g(X) = X^{ 12}  - s (X + 1)
\]
and impose conditions on $t,s$
that guarantee the Galois group of $fg$ to embed in 
%
%
\[
\widehat \Rub = (\Cyc{3}\wr \S_8)\times_{\sgn}(\Cyc{2}\wr \S_{12})^\circ.
\]
Let $e=(1,\ldots,1)\in \Cyc{3}^8$. 
Since $8\equiv -1 \bmod 3$, the map 
\begin{align*}
(\Cyc{3}\wr \S_8) \to (\Cyc{3}\wr \S_8)^\circ, &&
( x, \rho)\mapsto \pare{  x - \pare{\frac{1}{8}\sum_{i=1}^8 x_i }e ,\rho}.
\end{align*}
induces the following split exact sequence
$$ 1\to \Cyc{3}\to \widehat \Rub \to \Rub \to 1. $$

Since $\Rub$ is a quotient of $\widehat\Rub$, one can show the existence of a parametric $\Rub$-family by taking the corresponding 
$\Cyc{3}$ fixed field of the $\widehat \Rub$-family.
The conditions needed are as follows:
\begin{enumerate}
    \item $g(0)\in (\Q^\times)^2$, and
    \item $\disc{ f(X)} \equiv 
    \disc{g(X)} \bmod{(\Q^\times)^2}.$
\end{enumerate}

Computing the discriminants leads to
\[
-t^7(8^8 + 7^7t)  \equiv -s^{11}(12^{12} +11^{11}s )\bmod{(\Q^\times)^2}, 
\]
which turns out to be equivalent to 
\[
t\pare{1 + \frac{7^7}{8^8}t}  \equiv s\pare{1 +\frac{11^{11}} {12^{12}}s}\bmod{(\Q^\times)^2}.
\]
This becomes
\begin{align}
\label{eq:last}
(7 + \Tilde{t}^{-1} ) & \equiv (11 + \Tilde{s}^{-1} ) \bmod{(\Q^\times)^2}
\end{align}
for 
$$
\Tilde{t} = \frac{7^6}{8^8}t \quad\text{and}\quad  
\Tilde{s} = \frac{11^{10}}{12^{12}}s.
$$
The first condition makes $g(0) = -s$ a perfect square. This is equivalent to
$\Tilde{s}  = -u^2$ for $u\in\Q^\times.$
Putting this back into~\eqref{eq:last}, we get
\begin{align*}
    t = -\frac{8^8}{7^6} {(7 + (1-11u^2)v^2)}^{ -1 }
    \quad \text{ and }\quad 
    s = -u^2 \frac{12^{12}}{11^{10}}.
\end{align*}

Therefore
\begin{align*}
    f(X) = X^8 -t (X + 1), && 
    g(X) = X^{ 12 } - s (X + 1), 
\end{align*}
with
\begin{align}
    (t,s) = \left( \frac{ -8^8}{7^6(7+(1-11u^2)v^2)}, 
    \frac{    -12^{12}}{11^{10}}u^2   \right),
\end{align}
give a polynomial
\begin{align}\label{eq:pRubhat}
    p(u,v,X)=
    (X^2-X)^8f(t,(X^3-3X+1)/(X^2-X)) g(s,X^2)
\end{align}
whose splitting field over $\Q(u,v)$
has Galois group
$$
\widehat \Rub = 
(\Cyc{3}\wr \S_8)\times_{\sgn}(\Cyc{2}\wr \S_{12})^\circ$$
as is easily seen 
by Lemma~\ref{lem:spec} 
at $u=v=1$, together with the aid of MAGMA
(as in Subsection~\ref{ss:final_step}) .


\begin{proof}[Proof of Theorem \ref{thm:parametric}]
    We want to find a family of polynomials 
    $$h(u,v,X) \in \Q(u,v)[X]$$ 
    such that
    \[
    \Galf{h}{\Q(u,v)} \simeq \Rub.
    \]
    Consider the splitting field $F/\Q(u,v)$
    of $p$
    from \eqref{eq:pRubhat}.
    Since $\Rub\simeq \widehat{\Rub}/\pare{{\Cyc{3}}},$ one can take for $h$ the minimal polynomial of any $\gamma \in F$ generating the fixed field $F^{\pare{\Cyc{3}}}$ over $\Q(u,v)$.
\end{proof}

\section*{Acknowledgments}

We thank the ICTP for holding the 
School and Workshop on Number Theory and Physics in 2024, and 
Benjamin Sambale for comments and corrections. 
Mereb was partially supported by PIP 2021--2024 11220210100220CO (Conicet). 
Vendramin was supported in part by OZR3762
of Vrije Universiteit Brussel
and FWO Senior Research Project G004124N.


\bibliographystyle{abbrv}
\bibliography{refs}









\end{document}